\documentclass{amsart}
\usepackage{amssymb}
\usepackage{amsfonts}
\usepackage{pb-diagram,lamsarrow,pb-lams}
\usepackage{graphicx}
\usepackage{float}

\newtheorem{theorem}{Theorem}[section]
\theoremstyle{plain}

\newtheorem{corollary}[theorem]{Corollary}

\newtheorem{definition}[theorem]{Definition}

\newtheorem{lemma}[theorem]{Lemma}

\newtheorem{proposition}[theorem]{Proposition}

\numberwithin{equation}{section}

\theoremstyle{definition}

\DeclareMathOperator{\Hom}{Hom}
\DeclareMathOperator{\im}{im}
\DeclareMathOperator{\coker}{coker}

\DeclareMathOperator{\id}{id}
\newcommand{\Bl}{{\mathcal{B\ell}}}
\newcommand{\A}{\mathcal{A}}
\newcommand{\K}{\mathcal{K}}
\newcommand{\R}{\mathcal{R}}
\newcommand{\Q}{\mathbb{Q}}
\newcommand{\Z}{\mathbb{Z}}

\begin{document}
\title[Higher-Order Linking Forms for 3-Manifolds]{Higher-Order Linking Forms for
3-Manifolds}
\author{Constance Leidy$^{\dag}$\\
Wesleyan University}
\address{Wesleyan University, 655 Exley, 265 Church Street, Middletown, CT 06459}
\email{cleidy@wesleyan.edu}
\thanks{$^{\dag}$Partially supported by NSF DMS-1105776}

\begin{abstract}
Given a closed, oriented, connected 3-manifold, $M$, we define higher-order
linking forms on the higher-order Alexander modules of $M$. These higher-order linking forms generalize similar linking forms for knots previously studied by the author, which were themselves generalizations of the classical Blanchfield linking form for a knot. We also investigate the effect of the construction known as ``infection by a knot'' on these linking forms.
\end{abstract}


\maketitle

\section{Introduction}

We define linking forms, $\Bl_\R(M)$, associated to any closed, oriented, connected 3-manifold, $M$, and any Ore domain, $\R$, such that $\Z\Gamma \subset \R \subset \K\Gamma$ where $\phi:\pi_1(M) \to \Gamma$ is a coefficient system, such that $\Gamma$ is poly-torsion-free-abelian. Such linking forms have been used in a number of papers (see \cite{CHLlink}, \cite{CHLknot}, \cite{CHLderivatives}, \cite{CHLtorsion}, and \cite{CHLprimary}). However, the technical definitions and properties of them (particularly, the effect of infection on them) have not previously appeared in the literature.

Higher-order Alexander modules and higher-order linking forms for knots
and for closed 3-manifolds with $\beta_{1}(M)=1$ were introduced in \cite{cot1} and further
developed in \cite{nckt} and \cite{leidy}. Higher-order Alexander modules for 3-manifolds in general were defined and investigated in \cite{harvey}. In Section \ref{defsection}, we define higher-order linking forms for 3-manifolds which are defined on these higher-order Alexander module.

It should be pointed out that the coefficients that we consider are more general than those used in much of the previous related work. First of all, we allow our coefficients to be \emph{unlocalized}. In particular, the modules on which our linking forms are defined might \emph{not} have homological dimension 1 and the forms themselves might be \emph{singular}. This differs from much of the previous work (for instance, \cite{cot1} and \cite{cot2}) where the primary focus of study was over coefficients that were localized in order to obtain a principal ideal domain. Moreover, we allow $\Gamma$ to be an arbitrary poly-torsion-free-abelian group. Some of the previous work (for instance, \cite{harvey} and \cite{leidy}) focused on the case where $\Gamma=\pi_{1}(M)/\pi_{1}(M)^{(n)}_{r}$, the quotient of the fundamental group by the $n$th term of the (rational) derived series.

In Section \ref{geninfsection}, we investigate the effect of the construction known as ``infection by a knot'' on these higher-order linking forms for 3-manifolds. The construction of infecting a knot by a knot has been used extensively (for example, see \cite{cot1}, \cite{cot2} and \cite{nckt}). The effect of this construction on the higher-order Alexander modules of knots was studied in \cite{nckt}. The effect on the higher-order linking forms for knots was studied in \cite{leidy}. Infecting a 3-manifold by a knot was defined in \cite{harveycobordism}.

\section{Definition of Higher-Order Linking Forms for 3-manifolds}\label{defsection}

In order to define our linking forms, we will need a coefficient
system that embeds in its right ring of quotients. (A right ring of
quotients is the non-commutative analogue of a quotient field.) It
was shown in \cite{cot1} that the group rings of a certain class of
groups, namely poly-torsion-free-abelian groups have this property.

\begin{definition}A group $\Gamma$ is
poly-torsion-free-abelian (PTFA) if it admits a normal series $1 =
G_n \vartriangleleft G_{n-1} \vartriangleleft \ldots
\vartriangleleft G_0=\Gamma$ of subgroups such that the factors
$G_{i}/G_{i+1}$ are torsion-free abelian.
\end{definition}

\begin{proposition}[\cite{cot1}, Prop. 2.5]
If $\Gamma$ is PTFA, it follows that $\Z\Gamma$ is an Ore domain,
and therefore it is possible to define the right ring of fractions
of $\Z\Gamma$.
\end{proposition}

Suppose $\phi:\pi_1(M) \to \Gamma$ is a coefficient system, where
$\Gamma$ is PTFA. Then $\Z\Gamma$ has a right ring of fractions,
which we will denote by $\K\Gamma$. This right ring of fractions, $\K\Gamma$, is always a flat $\Z\Gamma$-module. (See \cite{stenstrom}, Prop. II.3.5.) If $\R$ is an Ore domain such
that $\Z\Gamma \subset \R \subset \K\Gamma$, then $\K\Gamma$ is also
the right ring of fractions of $\R$. (Such $\R$ could be $\Z\Gamma$
itself or could result from localizing any Ore set of $\Z\Gamma$.)

\begin{theorem}\label{def}Suppose $M$ is a closed, connected, oriented 3-manifold and
$\phi:\pi_1(M) \to \Gamma$ is a PTFA coefficient system. If $\R$ is
an Ore domain such that $\Z\Gamma \subset \R \subset \K\Gamma$, then
there is a linking form defined on the torsion submodule of
$H_1(M;\R)$:
$$\Bl_\R:TH_1(M;\R) \rightarrow \left(TH_1(M;\R) \right)^{\#}.$$
\end{theorem}
Here we use $\mathcal{M}^{\#}$ to denote $\overline{\Hom_{\R} \left(
\mathcal{M}, \K\Gamma/\R \right)}$. Also given any left $R$-module
$\mathcal{M}$, we use $\overline{\mathcal{M}}$ to denote the usual
associated right $R$-module resulting from the involution of $R$. The module $TH_1(M;\R)$ on which $\Bl_{\R}$ is defined is referred to as a \emph{higher-order Alexander module} of $M$. (Such modules were defined and studied in \cite{harvey}, where the focus was on the case where $\Gamma=\pi_{1}(M)/\pi_{1}(M)^{(n)}_{r}$.)

\begin{proof}
The short exact sequence $0 \rightarrow \R \rightarrow \K\Gamma
\rightarrow \K\Gamma/\R \rightarrow 0$ gives rise to the Bockstein
sequence of right $\R$-modules:
$$H_2(M;\K\Gamma) \overset{\psi}\rightarrow H_2(M;\K\Gamma/\R)
\overset{B}{\rightarrow} H_1(M;\R) \rightarrow H_1(M;\K\Gamma).$$
Since $\K\Gamma$ is a flat $\R$-module, $TH_1(M;\R)$ is the kernel
of the map $H_1(M;\R) \to H_1(M;\R)\otimes_{\R}\K\Gamma \cong
H_1(M;\K\Gamma)$. Using the Bockstein sequence above, we have
$TH_1(M;\R) = \im B \cong \coker \psi$. Hence in order to define
$\Bl_\R$ on $TH_1(M;\R)$, it suffices to define a map on
$H_2(M;\K\Gamma/\R)$ such that $\im\psi$ is in the kernel.

Consider the following commutative diagram of right $\R$-modules.
\renewcommand{\dgeverylabel}{\displaystyle}
$$
\begin{diagram}\dgARROWLENGTH=1em
\node{H_2(M;\K\Gamma)} \arrow[2]{e,t}{\psi} \arrow{s,r}{\text{P.D.}} \node{} \node{H_2(M;\K\Gamma/\R)} \arrow{s,r}{\text{P.D.}} \\
\node{\overline{H^1(M;\K\Gamma)}} \arrow[2]{e} \arrow{s,r}{\kappa} \node{} \node{\overline{H^1(M;\K\Gamma/\R)}} \arrow{s,r}{\kappa} \\
\node{\overline{\Hom_{\R}(H_1(M;\R),\K\Gamma)}} \arrow[2]{e} \arrow{s,r}{j^{\#}} \node{} \node{\overline{\Hom_{\R}(H_1(M;\R),\K\Gamma/\R)}} \arrow{s,r}{j^{\#}}  \\
\node{\overline{\Hom_{\R}(TH_1(M;\R),\K\Gamma)}} \arrow[2]{e} \node{} \node{\overline{\Hom_{\R}(TH_1(M;\R),\K\Gamma/\R)}} \\
\end{diagram}
$$

\vspace{-.25in}\noindent Here P.D. is the Poincar\'{e} duality isomorphism, $\kappa$ is the
Kronecker evaluation map, and $j^{\#}$ is induced by the inclusion
map.

Since $\K\Gamma$ is a torsion-free $\R$-module, it follows that
$\Hom_{ \R} \left( TH_1(M;\R), \K\Gamma \right) = 0$. In other words, the lower left corner of the above diagram is 0. Therefore the
image of $\psi$ is in the kernel of the composition $ j^{\#} \circ
\kappa \circ \text{P.D.}$. Hence, there is a well-defined map,
$\Bl_\R$, such that the following diagram is commutative.
\renewcommand{\dgeverylabel}{\displaystyle}
$$
\begin{diagram}
\node{H_2(M;\K\Gamma/\R)} \arrow[2]{e,t}{B} \arrow{s,l}{\text{P.D.}} \node{} \node{TH_1(M;\R)} \arrow{sssww,b}{\Bl_\R} \\
\node{\overline{H^1(M;\K\Gamma/\R)}} \arrow{s,l}{\kappa} \node{} \node{}\\
\node{\left(H_1(M;\R) \right)^{\#}} \arrow{s,l}{j^{\#}} \node{} \node{}\\
\node{\left(TH_1(M;\R) \right)^{\#}} \node{} \node{}\\
\end{diagram}
$$
\end{proof}

\section{The effect of infection by a knot on $\Bl_{\R}$}\label{geninfsection}

In this section, we consider the effect of infection by a knot on these higher-order linking forms. Let $M$ be a closed, connected, oriented 3-manifold, and let $\eta$ be an embedded, oriented, nullhomologous circle in $M$.
Then $\eta$ has a well-defined meridian, $\mu_{\eta}$, and
longitude, $\ell_{\eta}$. Delete the interior of a tubular
neighborhood of $\eta$. Replace it with the exterior, $E(J)$ of some
knot $J$ in $S^{3}$, identifying $\mu_{\eta}$ with the reverse of the longitude
$\ell_J$ of $J$, and $\ell_{\eta}$ with the meridian $\mu_J$ of $J$.
Denote the result $M(\eta, J)$, \emph{the result of infecting $M$ by $J$ along $\eta$}.

Let $\phi:\pi_1(M) \to \Gamma$ be a PTFA
coefficient system, and $\R$ be an Ore domain such that $\Z\Gamma
\subset \R \subset \K\Gamma$. Since there is a degree one map (rel boundary) $f:E(J) \to
E(\text{unknot})$, there is a degree one map from $M(\eta,J)$ to
$M$, which is the identity outside of $E(J)$. Hence the following
composition of maps defines coefficient systems on $E(J)$, $M(\eta,
J)$, and $M$:
$$\pi_1(E(J)) \overset{i_*}{\to} \pi_1(M(\eta,J)) \overset{f_*}{\to}
\pi_1(M) \overset{\phi}{\to} \Gamma.$$

First, we investigate the effect of infecting a 3-manifold by a knot on the higher-order Alexander modules, $TH_1(M;\R)$, on which the higher-order linking forms, $\Bl_{\R}(M)$, are defined. The effect of infecting a knot by a knot on the higher-order Alexander modules of knots was studied in Section 8 of \cite{nckt}.

\begin{proposition}\label{modules}
If $\phi(\eta)=1$, then $H_1(M(\eta,J);\R) \cong H_1(M;\R)$. If
$\phi(\eta) \neq 1$, then $H_1(M(\eta,J);\R) \cong H_1(M;\R) \oplus
H_1(E(J);\R)$.
\end{proposition}

\begin{proof}We begin by stating and proving the following necessary
lemma.

\begin{lemma}\label{E(J)Lemma}If $\phi(\eta)=1$, then $H_*(E(J);\R) \cong
H_*(E(J);\Z) \otimes_{\Z} \R$. If $\phi(\eta) \neq 1$, then
$H_*(E(J);\R) \cong H_*(E(J);\Z[t,t^{-1}]) \otimes_{\Z[t,t^{-1}]}
\R$, where $\R$ is a left $\Z[t,t^{-1}]$-module by the homomorphism
$t \mapsto \phi(\eta)$.
\end{lemma}

\begin{proof}Let $M(\eta)$ denote the result of deleting the interior of a
tubular neighborhood of $\eta$ from $M$. By the Seifert-Van Kampen
Theorem, we have the following presentations of $\pi_1(M(\eta,J))$
and $\pi_1(M)$:
$$\pi_1(M(\eta,J))=\langle \pi_1(M(\eta)), \pi_1(E(J)) |
\mu_{\eta}=\ell_J^{-1}, \ell_{\eta}=\mu_J \rangle$$
$$\pi_1(M)=\langle \pi_1(M(\eta)), t |
\mu_{\eta}=1, \ell_{\eta}=t \rangle$$ The map $f_*:\pi_1(M(\eta,J))
\to \pi_1(M)$ is the identity map on $\pi_1(M(\eta))$ and is the
Hurewicz map on $\pi_1(E(J)) \to \Z \cong \left<t\right>$ which
sends $\ell_J \mapsto 1$ and $\mu_J \mapsto t$. Therefore the map
$\phi \circ f_* \circ i_*: \pi_1(E(J)) \to \Gamma$ that defines the
coefficient system on $E(J)$ factors through the Hurewicz map, and
thus we have the following commutative diagram:
$$
\begin{diagram}\dgARROWLENGTH=1em
\node{\Z\pi_1(E(J))} \arrow{e} \arrow{s} \node{\R} \\
\node{\Z[t,t^{-1}]}  \arrow{ne,b}{\psi} \\
\end{diagram}
$$

\vspace{-.35in}\noindent Here $\psi: t \mapsto \phi(\eta)$.

If $\phi(\eta) \neq 1$, then $\psi$ is a monomorphism. It follows
from \cite[Lemma 1.3]{passman} that $\R$ is a free, and therefore
flat $\Z[t,t^{-1}]$-module. If $C_*(E(J);\Z\pi_1)$ denotes the chain
complex of the universal cover of $E(J)$ with the action of
$\Z\pi_1(E(J))$ on it, then we have:
\begin{eqnarray*}
H_*(E(J);\R) &=&H_*(C_*(E(J);\Z\pi_1) \otimes_{\Z\pi_1(E(J))} \R) \\
&\cong& H_*(C_*(E(J);\Z\pi_1) \otimes_{\Z\pi_1(E(J))} \Z[t,t^{-1}] \otimes_{\Z[t,t^{-1}]} \R) \\
&\cong& H_*(E(J);\Z[t,t^{-1}]) \otimes_{\Z[t,t^{-1}]} \R. \\
\end{eqnarray*}

If $\phi(\eta)=1$, then $\psi$ further factors through $\Z$:

$$
\begin{diagram}\dgARROWLENGTH=1em
\node{\Z\pi_1(E(J))} \arrow{e} \arrow{s} \node{\R} \\
\node{\Z}     \arrow{ne,b}{\widehat{\psi}}    \\
\end{diagram}
$$

\vspace{-.35in}\noindent Since $\widehat{\psi}$ is a monomorphism, it follows that $\R$ is a
free and therefore flat $\Z$-module. By an argument analogous to
that above, $H_*(E(J);\R) \cong H_*(E(J);\Z) \otimes_{\Z} \R$.
\end{proof}

We now continue with the proof of Proposition \ref{modules}.
Consider the Mayer-Vietoris sequence for $M(\eta,J) \cong E(J)
\cup_{\partial E(J)} M(\eta)$:
\renewcommand{\dgeverylabel}{\displaystyle}
$$
\begin{diagram}
\node{H_1(M(\eta,J);\R)} \arrow{e,t}{\partial_*} \node{H_0(\partial E(J);\R)}
\arrow{e,t}{(\psi_1,\psi_2)} \node{H_0(E(J);\R) \oplus H_0(M(\eta);\R).} \\
\end{diagram}
$$

\vspace{-.45in}Since $H_0(\partial E(J);\Z) \cong H_0(E(J);\Z)$ and $H_0(\partial
E(J);\Z[t,t^{-1}]) \cong H_0(E(J);\Z[t,t^{-1}])$, it follows from
Lemma \ref{E(J)Lemma} that $\partial_*$ is the trivial map.

Since infecting by the unknot, $U$, leaves the manifold unchanged,
we have the following commutative diagram of $\R$-modules where the
rows are Mayer-Vietoris exact sequences:
\renewcommand{\dgeverylabel}{\displaystyle}
$$
\begin{diagram}\dgARROWLENGTH=2.5em
\node{H_1(\partial E(J))} \arrow{e,t}{(\psi_1,\psi_2)} \arrow{s,r}{f_*} \node{H_1(E(J)) \oplus H_1(M(\eta))} \arrow{e,t}{} \arrow{s,r}{f_*} \node{H_1(M(\eta,J))} \arrow{e,t}{\partial_*} \arrow{s,r}{f_*} \node{0} \\
\node{H_1(\partial E(J))} \arrow{e,t}{(\overline{\psi_1},\psi_2)} \node{H_1(E(U))\oplus H_1(M(\eta))} \arrow{e,t}{} \node{H_1(M)} \arrow{e,t}{\partial_*} \node{0.}\\
\end{diagram}
$$

\vspace{-.5in}Suppose $\phi(\eta)=1$. Since $H_1(\partial E(J);\Z) \to
H_1(E(J);\Z)$ is an epimorphism, by Lemma \ref{E(J)Lemma}, $\psi_1$
is an epimorphism . Hence $H_1(M(\eta,J);\R) \cong
H_1(M(\eta);\R)/\im(\psi_2)$. Similarly, $H_1(M;\R) \cong
H_1(M(\eta);\R)/\im(\psi_2)$. Therefore, $H_1(M(\eta,J);\R) \cong
H_1(M;\R)$.

Suppose $\phi(\eta) \neq 1$. Since $\mu_J$ unwinds and $\ell_J$
bounds a lift of the Seifert surface in the infinite cyclic cover,
$H_1(\partial E(J);\Z[t,t^{-1}]) \to H_1(E(J);\Z[t,t^{-1}])$ is the
zero map. By Lemma \ref{E(J)Lemma}, it follows that $\psi_1$ is the
zero map. Hence $H_1(M(\eta,J);\R) \cong H_1(E(J);\R) \oplus
H_1(M(\eta);\R)/\im(\psi_2)$. Furthermore, since
$H_1(E(U);\Z[t,t^{-1}])=0$, it follows that $H_1(E(U);\R)=0$. Hence
$H_1(M;\R) \cong H_1(M(\eta);\R)/\im(\psi_2)$. Therefore,
$H_1(M(\eta,J);\R) \cong H_1(E(J);\R) \oplus H_1(M;\R)$.
\end{proof}

\begin{corollary}\label{modules-torsion}
If $\phi(\eta) \neq 1$, then $$TH_1(M(\eta,J);\R) \cong TH_1(M;\R)
\oplus H_1(E(J);\R) \cong TH_1(M;\R) \oplus \left(\A_0(J)
\otimes_{\Z[t,t^{-1}]} \R \right),$$ where $\A_0(J)$ is the
classical Alexander module of $J$.
\end{corollary}

\begin{proof}
Since $\A_0(J) = H_1(E(J);\Z[t,t^{-1}])$ is annihilated by the
Alexander polynomial, it follows that $\left(\A_0(J)
\otimes_{\Z[t,t^{-1}]} \R \right)$ is a torsion module. The result
now follows from Proposition \ref{modules} and Lemma
\ref{E(J)Lemma}.
\end{proof}

We now consider the effect of infecting a 3-manifold by a knot on the higher-order linking forms for 3-manifolds. The effect of infecting a knot by a knot on the higher-order linking forms for knots was shown in Section 4 of \cite{leidy}.

\begin{proposition}\label{lowerforms}
If $\phi(\eta)=1$, then the linking forms
$\Bl_\R(M(\eta,J)):TH_1(M(\eta,J);\R) \rightarrow
\left(TH_1(M(\eta,J);\R) \right)^{\#}$ and $\Bl_\R(M):TH_1(M;\R)
\rightarrow \left(TH_1(M;\R) \right)^{\#}$ are isomorphic.
\end{proposition}

\begin{proof}Recall that there is a degree one map $f: M(\eta,J)\to M$. By Proposition \ref{modules}, $f$ induces an isomorphism between
$TH_1(M(\eta,J);\R)$ and $TH_1(M;\R)$. 

\pagebreak

We have the following
commutative diagram:
\renewcommand{\dgeverylabel}{\displaystyle}

$$
\hspace{-.25in}\begin{diagram}\dgARROWPARTS=9
\node{H_2(M(\eta,J);\K\Gamma/\R)} \arrow[3]{e,t,6}{f_*} \arrow[2]{se,t,7}{B} \arrow{s,l}{P.D.} \node{} \node{} \node{H_2(M;\K\Gamma/\R)} \arrow[2]{se,t,7}{B} \arrow{s,l}{P.D.}\node{} \node{}\\
\node{\overline{H^1(M(\eta,J);\K\Gamma/\R)}} \arrow[2]{s,l}{\kappa} \node{} \arrow{w}  \node{} \node{\overline{H^1(M;\K\Gamma/\R)}} \arrow[2]{w,t,-}{f^*} \arrow{s,-} \node{} \node{} \\
\node{} \node{} \node{TH_1(M(\eta,J);\R)} \arrow[3]{e,t}{f_*} \arrow[2]{sw,b,6}{\Bl_\R(M(\eta,J))} \node{} \arrow{s,l}{\kappa} \node{} \node{TH_1(M;\R)} \arrow[2]{sw,b,6}{\Bl_\R(M)}\\
\node{\left(H_1(M(\eta,J);\R)\right)^{\#}} \arrow{s,l}{j^{\#}} \node{} \arrow{w} \node{} \node{\left(H_1(M;\R)\right)^{\#}} \arrow[2]{w,t,-}{f^*} \arrow{s,l}{j^{\#}}  \node{} \node{} \\
\node{\left(TH_1(M(\eta,J);\R)\right)^{\#}} \node{} \node{} \node{\left(TH_1(M;\R)\right)^{\#}} \arrow[3]{w,t,3}{f^*} \node{} \node{} \\
\end{diagram}
$$

\vspace{-.5in}Therefore $\Bl_\R(M(\eta,J))=f^{*} \circ \Bl_\R(M) \circ f_*$, and
hence $\Bl_\R(M(\eta,J))$ and $\Bl_\R(M)$ are isomorphic.
\end{proof}

In the remainder of this section, we show how the linking forms
$\Bl_{\R}(M)$ and $\Bl_{\R}(M(\eta,J))$ are related when $\phi(\eta)
\neq 1$. We begin by defining a linking form on $E(J)$ with
coefficients that are compatible with viewing $J$ as the infecting
knot of an infection.

\begin{proposition}
If $\phi(\eta) \neq 1$, then for any knot $J$, there is a linking
form $\Bl_\R(J):H_1(E(J);\R) \rightarrow \left(H_1(E(J);\R)
\right)^{\#}$ where the coefficient system is induced by the
composition $\pi_1(E(J)) \overset{i_*}{\to} \pi_1(M(\eta,J))
\overset{f_*}{\to} \pi_1(M) \overset{\phi}{\to} \Gamma$.
\end{proposition}

\begin{proof}
We consider the Bockstein sequence:
$$H_2(E(J);\K\Gamma) \rightarrow H_2(E(J);\K\Gamma/\R)
\overset{B}{\rightarrow} H_1(E(J);\R) \rightarrow
H_1(E(J);\K\Gamma).$$ From Lemma \ref{E(J)Lemma}, we have that
$H_1(E(J);\R) \cong H_1(E(J);\Z[t,t^{-1}]) \otimes_{\Z[t,t^{-1}]}
\R$. Since $H_1(E(J);\Z[t,t^{-1}])=\A_0(J)$ is annihilated by the
Alexander polynomial, it follows that $H_1(E(J);\R)$ is a torsion
module. Hence $H_1(E(J);\K\Gamma)=0$, and by Poincar\'{e} duality,
$H_2(E(J);\K\Gamma)=0$. Therefore the map $B$ above is an
isomorphism. We define the linking form $\Bl_\R(J)$ to be the
composition of the following maps:
\begin{eqnarray*}
&&H_1(E(J);\R) \overset{B^{-1}}{\rightarrow}
H_2(E(J);\K\Gamma/\R) \overset{P.D.}{\rightarrow} \overline{H^1(E(J),\partial E(J);\K\Gamma/\R)} \\
&&\hspace{0.5in}\overset{\pi^*}{\rightarrow}
\overline{H^1(E(J);\K\Gamma/\R)} \overset{\kappa}{\rightarrow}
H_1(E(J);\R)^{\#},
\end{eqnarray*}
where $P.D.$ is the Poincar\'{e} duality isomorphism, $\pi^*$ is the
map in the long exact sequence of a pair and $\kappa$ is the
Kronecker evaluation map.
\end{proof}

\pagebreak

We now show that $\Bl_{\R}(J)$ is determined by the classical
Blanchfield linking form on $J$. In the proof of Lemma
\ref{E(J)Lemma}, we considered the following commutative diagram:
$$
\begin{diagram}\dgARROWLENGTH=1em
\node{\Z\pi_1(E(J))} \arrow{e} \arrow{s} \node{\R} \\
\node{\Z[t,t^{-1}]} \arrow{ne,b}{\psi}\\
\end{diagram}
$$

\vspace{-.5in}\noindent Here $\psi: t \mapsto \phi(\eta)$. If $\phi(\eta) \neq 1$, then
$\psi$ and $\overline{\psi}: \Q(t)/\Z[t,t^{-1}] \to \K\Gamma/\R$ are
monomorphisms. Furthermore we have a map $\psi_*: \A_0(J) =
H_1(E(J);\Z[t,t^{-1}]) \to H_1(E(J);\R)$.

\begin{proposition}\label{classicalBl}
If $\phi(\eta) \neq 1$, then  for all $x, y \in \A_0(J)$,
$$\Bl_{\R}(J)(\psi_*(x),\psi_*(y)) = \overline{\psi}\left(\Bl_0(J)(x,y)\right),$$
where $\Bl_0(J)$ is the classical Blanchfield linking form on $J$.
\end{proposition}

\begin{proof}The classical Blanchfield linking form on $J$ is the
composition of the following maps:
\begin{eqnarray*}
&&H_1(E(J);\Z[t,t^{-1}]) \overset{B^{-1}}{\rightarrow}
H_2(E(J);\Q(t)/\Z[t,t^{-1}]) \overset{P.D.}{\rightarrow} \overline{H^1(E(J),\partial E(J);\Q(t)/\Z[t,t^{-1}])} \\
&&\hspace{0.5in}\overset{\pi^*}{\rightarrow}
\overline{H^1(E(J);\Q(t)/\Z[t,t^{-1}])}
\overset{\kappa}{\rightarrow}
\overline{\Hom_{\Z[t,t^{-1}]}(H_1(E(J);\Z[t,t^{-1}]),\Q(t)/\Z[t,t^{-1}]),}
\end{eqnarray*}
where $P.D.$ is the Poincar\'{e} duality isomorphism, $\pi^*$ is the
map in the long exact sequence of a pair and $\kappa$ is the
Kronecker evaluation map. 

\pagebreak

We have the following commutative diagram:
\renewcommand{\dgeverylabel}{\displaystyle}

$$
\hspace{-1.25in}\begin{diagram}\dgARROWLENGTH=1em
\node{H_1(E(J);\Z[t,t^{-1}])} \arrow[2]{e,t}{\psi_*} \arrow{s,l}{B^{-1}} \node{} \node{H_1(E(J);\R)} \arrow{s,l}{B^{-1}}\\
\node{H_2(E(J);\Q(t)/\Z[t,t^{-1}])} \arrow[2]{e,t}{\overline{\psi}_*} \arrow{s,l}{P.D.} \node{} \node{H_2(E(J);\K\Gamma/\R)} \arrow{s,l}{P.D.}\\
\node{\overline{H^1(E(J),\partial E(J);\Q(t)/\Z[t,t^{-1}])}} \arrow[2]{e,t}{\overline{\psi}_*} \arrow{s,l}{\pi^*} \node{} \node{\overline{H^1(E(J),\partial E(J);\K\Gamma/\R)}} \arrow{s,l}{\pi^*}\\
\node{\overline{H^1(E(J);\Q(t)/\Z[t,t^{-1}])}}\arrow[2]{e,t}{\overline{\psi}_*} \arrow{s,l}{\kappa} \node{} \node{\overline{H^1(E(J);\K\Gamma/\R)}} \arrow{s,l}{\kappa}\\
\node{\overline{\Hom_{\Z[t,t^{-1}]}(H_1(E(J);\Z[t,t^{-1}]),\Q(t)/\Z[t,t^{-1}])}} \arrow{se,l}{\overline{\psi}_{\#}} \node{} \node{\overline{\Hom_{\R}(H_1(E(J);\R),\K\Gamma/\R)}} \arrow{sw,l}{\psi^*}\\
\node{} \node{\overline{\Hom_{\R}(H_1(E(J);\Z[t,t^{-1}]),\K\Gamma/\R)}} \node{} \\
\end{diagram}
$$

\vspace{-.5in}The composition of maps in the left column is the classical Blanchfield linking form $\Bl_{0}(J)$, and in the right column is $\Bl_{\R}(J)$.

Since the diagram commutes, $\psi^* \circ \Bl_{\R}(J) \circ \psi_* =
\overline{\psi}_{\#} \circ \Bl_0(J)$. Evaluating these maps on $x, y
\in \A_0(J)$, gives the desired result.
\end{proof}

We now show the relationship between the linking forms $\Bl_{\R}(M)$
and $\Bl_{\R}(M(\eta,J))$ when $\phi(\eta) \neq 1$. In this case, it
follows from Corollary \ref{modules-torsion} that the following is a
split short exact sequence:
$$H_1(E(J);\R) \overset{i_*}{\to} TH_1(M(\eta,J);\R) \overset{f_*}{\to} TH_1(M;\R).$$
If we choose a splitting $g$, we have the following theorem that
relates $\Bl_{\R}(M(\eta,J))$, $\Bl_{\R}(M)$, and $\Bl_{\R}(J)$.

\begin{theorem}
\label{DirectSum}If $\phi(\eta) \neq 1$, then $\Bl_{\R}(M(\eta,J))
\cong \Bl_{\R}(M) \oplus \Bl_{\R}(J)$. That is, for any $x_1,y_1 \in
TH_1(M;\R)$ and $x_2,y_2 \in H_1(E(J);\R)$,
\begin{equation*}
\Bl_{\R}(M)(x_1,y_1) + \Bl_{\R}(J)(x_2,y_2) =
\Bl_{\R}(M(\eta,J))\left(g(x_1) + i_*(x_2), g(y_1) +
i_*(y_2)\right).
\end{equation*}
\end{theorem}

Before giving the proof, we state a corollary that follows
immediately from Proposition \ref{classicalBl} and Theorem
\ref{DirectSum}.

\begin{corollary}If $\phi(\eta) \neq 1$, then for any $x_1,y_1 \in
TH_1(M;\R)$ and $x_2,y_2 \in \A_0(J)$,
\begin{equation*}
\Bl_\R(M)(x_1,y_1) + \overline{\psi}\left(\Bl_0(J)(x_2,y_2)\right) =
\Bl_\R(M(\eta,J))\left(g(x_1) + i_*(\psi_*(x_2)), g(y_1) +
i_*(\psi_*(y_1))\right).
\end{equation*}
\end{corollary}

From Corollary \ref{modules-torsion}, we know that every element in
$TH_1(M(\eta,J);\R)$ can be written as $g(x_1) + i_*(\psi_*(x_2))$
for some $x_1 \in TH_1(M;\R)$ and $x_2 \in \A_0(J)$. Hence the
corollary above shows that the linking form on $M(\eta,J)$ is
completely determined by the linking form on $M$ and the classical
Blanchfield linking form on $J$. We now prove Theorem
\ref{DirectSum}.

\begin{proof}
We have the following diagram.
\renewcommand{\dgeverylabel}{\displaystyle}

$$
\begin{diagram} 
\node{H_1(E(J);\R)} \arrow{e,t}{i_*} \arrow{s,r}{\Bl_{\R}(J)} \node{TH_1(M(\eta,J);\R)}  \arrow{s,r}{\Bl_{\R}(M(\eta,J))} \node{TH_1(M;\R)} \arrow{w,t}{g} \arrow{s,r}{\Bl_{\R}(M)}\\
\node{H_1(E(J);\R)^{\#}}  \node{TH_1(M(\eta,J);\R)^{\#}} \arrow{w,t}{i^*} \arrow{e,t}{g^{\#}} \node{TH_1(M;\R)^{\#}} \\
\end{diagram}
$$

\vspace{-.5in}\noindent where $g^{\#}$ is the dual of $g$. Notice that since $f_* \circ g =
\id$, it follows that $g^{\#} \circ f^* = \id$. The isomorphism in
the theorem will be given by $i_* \oplus g$. Hence the theorem will
follow from the following four claims.

\begin{enumerate}
\item $i^* \circ \Bl_{\R}(M(\eta,J)) \circ
i_* = \Bl_{\R}(J)$ which establishes: \\
$$\Bl_{\R}(M(\eta,J))(i_*(x_1),i_*(y_1)) = \Bl_{\R}(J) (x_1,y_1).$$

\item $g^{\#} \circ \Bl_{\R}(M(\eta,J)) \circ g =
\Bl_{\R}(M)$ which establishes: \\
$$\Bl_{\R}(M(\eta,J)) (g(x_2),g(y_2)) = \Bl_{\R}(M) (x_2,y_2).$$

\item $g^{\#} \circ \Bl_{\R}(M(\eta,J)) \circ i_* = 0$ which establishes: \\
$$\Bl_{\R}(M(\eta,J)) (i_*(x_1),g(y_2)) = 0.$$

\item $i^* \circ \Bl_{\R}(M(\eta,J)) \circ g = 0$
which establishes: \\
$$\Bl_{\R}(M(\eta,J)) (g(x_2),i_*(y_1)) = 0.$$
\end{enumerate}

\pagebreak

The first claim follows immediately from the following commutative diagram.
\renewcommand{\dgeverylabel}{\displaystyle}

$$
\hspace{-.5in}\begin{diagram}\dgARROWPARTS=5
\node{H_2(E(J);\K\Gamma/\R)} \arrow[4]{e,t,4}{i_*} \arrow{s,l}{P.D.} \arrow[3]{se,t}{B}\node{} \node{} \node{} \node{H_2(M(\eta,J);\K\Gamma/\R)} \arrow{sssee,t}{B} \arrow[2]{s,l}{P.D.} \node{} \node{} \\
\node{\overline{H^1(E(J),\partial E(J);\K\Gamma/\R)}} \arrow{s,l}{\pi^*} \node{} \node{} \node{} \node{} \node{} \node{} \\
\node{\overline{H^1(E(J);\K\Gamma/\R)}} \arrow[3]{s,l}{\kappa} \node{} \node{} \arrow[2]{w} \node{} \node{\overline{H^1(M(\eta,J);\K\Gamma/\R)}} \arrow[2]{w,t,-}{i^*} \arrow{s,-} \node{} \node{}\\
\node{} \node{} \node{} \node{H_1(E(J);\R)} \arrow[3]{e,t}{i_*} \arrow[3]{sw,b}{\Bl_\R(J)} \node{} \arrow[2]{s,l}{\kappa} \node{} \node{TH_1(M(\eta,J);\R)} \arrow{sssww,b}{\Bl_\R(M(\eta,J))} \\
\node{} \node{} \node{} \node{} \node{} \node{} \node{}\\
\node{\left(H_1(E(J);\R)\right)^{\#}} \arrow{s,l}{\id} \node{} \arrow{w} \node{} \node{} \node{\left(H_1(M(\eta,J);\R)\right)^{\#}} \arrow[3]{w,t,-}{i^*} \arrow{s,l}{j^{\#}} \node{} \node{}\\
\node{\left(H_1(E(J);\R)\right)^{\#}} \node{} \node{} \node{} \node{\left(TH_1(M(\eta,J);\R)\right)^{\#}} \arrow[4]{w,t,1}{i^*} \node{} \node{}\\
\end{diagram}
$$

\pagebreak

To prove the second claim, we consider the following commutative
diagram.
\renewcommand{\dgeverylabel}{\displaystyle}
$$
\begin{diagram}\dgARROWPARTS=6
\node{H_2(M(\eta,J);\K\Gamma/\R)} \arrow[3]{e,t,4}{f_*} \arrow[2]{se,t,2}{B} \arrow{s,l}{P.D.} \node{} \node{} \node{H_2(M;\K\Gamma/\R)} \arrow[2]{se,t,2}{B} \arrow{s,l}{P.D.} \node{} \node{}\\
\node{\overline{H^1(M(\eta,J);\K\Gamma/\R)}} \arrow[2]{s,l}{\kappa} \node{} \arrow{w} \node{} \node{\overline{H^1(M;\K\Gamma/\R)}} \arrow[2]{w,t,-}{f^*} \arrow{s,-} \node{} \node{}\\
\node{} \node{} \node{TH_1(M(\eta,J);\R)} \arrow[2]{sw,b,4}{\Bl_\R(M(\eta,J))} \arrow[3]{e,t}{f_*} \node{} \arrow{s,l}{\kappa} \node{} \node{TH_1(M;\R)}  \arrow[2]{sw,b,4}{\Bl_\R(M)}\\
\node{\left(H_1(M(\eta,J);\R)\right)^{\#}} \arrow{s,l}{j^{\#}} \node{} \arrow{w} \node{} \node{\left(H_1(M;\R)\right)^{\#}} \arrow{s,l}{j^{\#}} \arrow[2]{w,t,-}{f^*} \node{} \node{}\\
\node{\left(TH_1(M(\eta,J);\R)\right)^{\#}}  \node{} \node{} \node{\left(TH_1(M;\R)\right)^{\#}} \arrow[3]{w,t}{f^*} \node{} \node{}\\
\end{diagram}
$$

\vspace{-.5in}From the diagram above we have $f^* \circ \Bl_\R(M) \circ f_* =
\Bl_\R(M(\eta,J))$. Therefore,
$$g^{\#} \circ f^* \circ \Bl_\R(M) \circ f_* \circ g = g^{\#} \circ
\Bl_\R(M(\eta,J)) \circ g.$$ Since $f_* \circ g = \id$ and $g^{\#}
\circ f^* = \id$, it follows that $g^{\#} \circ \Bl_\R(M(\eta,J))
\circ g = \Bl_\R(M)$. Hence the second claim is proved.

We have established that we have the following commutative diagram whose rows are exact.
\renewcommand{\dgeverylabel}{\displaystyle}
$$
\begin{diagram} 
\node{H_1(E(J);\R)} \arrow{e,t}{i_*} \arrow{s,r}{\Bl_{\R}(J)} \node{TH_1(M(\eta,J);\R)} \arrow{s,r}{\Bl_{\R}(M(\eta,J))} \arrow{e,t}{f_*} \node{TH_1(M;\R)} \arrow{s,r}{\Bl_{\R}(M)}  \\
\node{H_1(E(J);\R)^{\#}}  \node{TH_1(M(\eta,J);\R)^{\#}} \arrow{w,t}{i^*} \node{TH_1(M;\R)^{\#}} \arrow{w,t}{f^*}    \\
\end{diagram}
$$

\vspace{-.5in}Since $f^* \circ \Bl_{\R}(M) \circ f_* = \Bl_{\R}(M(\eta,J))$, it
follows that $$g^{\#} \circ \Bl_{\R}(M(\eta,J)) \circ i_* = g^{\#}
\circ f^* \circ \Bl_{\R}(M) \circ f_* \circ i_*$$ $$i^* \circ
\Bl_{\R}(M(\eta,J)) \circ g = i^* \circ f^* \circ \Bl_{\R}(M) \circ
f_* \circ g$$

But since the rows are exact, $f_* \circ i_* = 0$ and $i^* \circ f^*
= 0$. Therefore $g^{\#} \circ \Bl_{\R}(M(\eta,J)) \circ i_* = 0$ and
$i^* \circ \Bl_{\R}(M(\eta,J)) \circ g = 0$.
\end{proof}

\bibliographystyle{plain}
\bibliography{3-mfld}

\end{document}